\documentclass[10pt]{amsart}

\usepackage{amsmath,amsthm,amsfonts, bbm, amssymb, hyperref,graphicx,color}

\newcommand{\ii}{\mathbbm{i}}
\newcommand{\rad}{\operatorname{rad}}

\newcommand{\C}{\mathbb{C}}

\newcommand{\Heis}{\mathbb{H}}

\newcommand{\Sieg}{\mathcal{S}}
\newcommand{\R}{\mathbb{R}}
\newcommand{\Z}{\mathbb{Z}}
\newcommand{\Q}{\mathbb{Q}}

\newcommand{\norm}[1]{\left\vert #1 \right \vert}	
\newcommand{\Norm}[1]{\left\Vert #1 \right \Vert}

\renewcommand{\Re}{\text{Re}}
\renewcommand{\Im}{\text{Im}}

\def\[#1\]{\begin{align}#1\end{align}}
\def\(#1\){\begin{align*}#1\end{align*}}

\newtheorem{thm}{Theorem}

\newtheorem{prop}[thm]{Proposition}
\newtheorem{lemma}[thm]{Lemma}
\newtheorem{cor}[thm]{Corollary}

\theoremstyle{definition}

\theoremstyle{definition}

\newtheorem{remark}[thm]{Remark}

\numberwithin{equation}{section}

\title[Diophantine properties of Heisenberg continued fractions]{Diophantine properties of continued fractions on the Heisenberg group}
\author[J. Vandehey]{Joseph Vandehey}
\address{
Department of Mathematics\\
University of Georgia at Athens\\
Athens, GA 30602
}
\email{vandehey@uga.edu}

\begin{document}

\begin{abstract}
We provide several results on the diophantine properties of continued fractions on the Heisenberg group, many of which are analogous to classical results for real continued fractions. In particular, we show an analog of Khinchin's theorem and that convergents are also best approximants up to a constant factor of the denominator.
\end{abstract}

\date{\today}

\maketitle

\section{Introduction}

\subsection{Heisenberg continued fractions and a brief summary of results}\label{sec:firstintro}

We consider the Heisenberg group in its Siegel model, given by the space
\(
\Sieg := \{h=(u,v)\in \mathbb{C}^2: |u|^2-2\operatorname{Re} (v)=0\}
\)
with group law given by 
\(
(u_1,v_1)*(u_2,v_2)=(u_1+u_2,v_1+\overline{u_1}u_2+v_2) \qquad (u,v)^{-1} = (-u,\overline{v}).
\)
In a previous paper \cite{LV}, Lukyanenko and the author studied a continued fraction algorithm on this space, which we define here in the following way.

We denote the set of integer points in $\Sieg$ by $\Sieg(\Z)=\Sieg\cap \Z[\ii]^2$, and define the norm of a point by $\Norm{(u,v)}=\norm{v}^{1/2}$. For almost all points $h$, there exists a unique nearest integer $[h]\in \Sieg(\Z)$ that minimizes $\Norm{[h]^{-1}*h}$. Let $K_D\subset \Sieg$ be the set of points $h\in \Sieg$ such that $[h]=(0,0)$ (with some appropriate choice of boundary). We also consider the Koranyi inversion $\iota$, a conformal map on $\Sieg$ given by $\iota(u,v)=(-u/v,1/v)$, which is the analog in this setting of the inversion map $x\mapsto 1/x$ for real continued fractions.

Let $T:K_D\to K_D$ be an analog of the Gauss map given by
\(
Th=\begin{cases}
(0,0), & \text{if }h=(0,0),\\
[\iota h]^{-1}*\iota h, & \text{otherwise.}
\end{cases}
\)
This closely resembles the classical Gauss map, which looks like $Tx = x^{-1}-\lfloor x^{-1}\rfloor$. Continuing the analogy to other important continued fraction concepts, given a point $h\in \Sieg$ we define the forward iterates, $h_i\in K_D$, and the continued fraction digits, $\gamma_i\in \Sieg(\Z)$, for $h$ by 
\(
\gamma_0&=[h] & h_0 &= [h]^{-1}*h =\gamma_0^{-1}*h \\
\gamma_i &= [\iota h_{i-1}] & h_i&=T^i h_0= \gamma_{i}^{-1}*h_{i-1},
\)
with the sequence of continued fraction digits terminating if $h_i=(0,0)$. The Gauss map acts on the sequence of digits via a forward shift. The convergents are then given by
\(
\left( \frac{r_n}{q_n}, \frac{p_n}{q_n}\right) := \gamma_0*\left( \iota \gamma_1*\left( \iota\gamma_2*\dots \gamma_n\right)\right),
\)
where $r_n$, $p_n$, and $q_n$ are relatively prime Gaussian integers.

In \cite{LV}, it was shown that these definitions do give a good notion of a continued fraction algorithm. In particular, the sequence of convergents $(r_n/q_n,p_n/q_n)$ does always converge to $h$, and the sequence of continued fraction digits is finite if and only if $h$ is a rational point (i.e., $h\in \Sieg(\Q)=\Sieg \cap \Q[\ii]^2$). So for an irrational point $h\in \Sieg \setminus \Sieg(\Q)$ it would make sense to write it as
\[
h=\gamma_0\iota \gamma_1\iota \gamma_2\cdots,
\]
suppressing parentheses and * notation. Along the way, it was also shown that these continued fractions satisfy some classical relations which no other multi-dimensional continued fraction algorithm appears to satisfy (see Proposition \ref{prop:distance}). This suggests that continued fractions are a natural object of study on the Heisenberg group. 

In this paper, we move from analyzing the basic algorithm to studying what these continued fractions can tell us about the Heisenberg group itself. We will examine how these continued fractions on $\Sieg$ relate to diophantine approximation on $\Sieg$. We will give explicit bounds on how well an irrational point $h$ is approximated by its convergents $(r_n/q_n,p_n/q_n)$, we will show that the rate of approximation given by the convergents is as fast as possible for almost all points $h$, and we will show that the convergents are closely related to the \emph{best approximations} to $h$.

These results are surprising for a few reasons. First, the rate of convergence (namely the Diophantine approximation exponent) is not what we expect from classical work on multi-dimensional continued fraction. (This may be a consequence of the notion of distance we use.) Second, as Schweiger says, algorithms of this type ``are not likely to provide ``good" approximations which satisfy the $n$-dimensional Dirichlet property," \cite[pp.~vi]{schweiger} yet we obtain a precise rate of convergence result. (See also the discussion on pages 129--130 of \cite{schweiger}.) Finally, these results, surprising in both their statement and existence, are obtained using elementary techniques, some of which are variants on those used for classical continued fractions. 

A good overview of multi-dimensional continued fractions can be found in \cite{schweiger} and a slightly longer discussion of their approximation properties can be found in \cite{Hensley}.

\subsection{Diophantine approximation with regular continued fractions}
The connection between diophantine approximation of real numbers and regular continued fraction (RCF) expansions is extremely classical. We recall some of the definitions and results here.

 If an irrational number $x$ has RCF expansion
\(
x = a_0+\cfrac{1}{a_1+\cfrac{1}{a_2+\dots}}, \qquad a_i \in \mathbb{N},
\)
then we say 
\(
\frac{p_n}{q_n} = a_0+\cfrac{1}{a_1+\cfrac{1}{a_2+\dots+ \cfrac{1}{a_n}}}
\)
is the $n$th convergent. The rate of approximation is given by
\[
\label{eq:RCFapprox}
\left| x- \frac{p_n}{q_n} \right| < \frac{1}{q_n^2}.
\]
(See for example Theorem 9 in \cite{Khinchin}.)

This relates to a natural diophantine question: given an arithmetic function $f(q)$, how many solutions are there in rational numbers $p/q$ to the inequality
\(
\left| x- \frac{p}{q} \right| < f(q)?
\)
Continued fractions can be used to show the following powerful diophantine result.

\begin{thm}\label{thm:Hurwitz}[Hurwitz's Theorem \cite{Hurwitz}]
For every irrational number $x$ there exist infinitely many pairs of relatively prime integers $(p,q)$, such that
\[\label{eq:hurwitz}
\left| x- \frac{p}{q} \right| < \frac{1}{\sqrt{5}} \frac{1}{q^2}.
\]
The constant $1/\sqrt{5}$ here is known as the Hurwitz constant and is best possible.
\end{thm}

In fact, at least one out of every three consecutive RCF convergents must satisfy \eqref{eq:hurwitz} (Theorem 20 in \cite{Khinchin}). Markov also showed that the constant in Hurwitz's theorem could be improved if one removed a countable set of irrational numbers that have similar continued fraction expansions, see \cite{Cassels, Nicholls}.

From \eqref{eq:RCFapprox} and Hurwitz's Theorem, we see that continued fraction convergents make for very good approximations. The next two theorems show (respectively) that any sufficiently good rational approximation must itself be an RCF convergent, and that the order of approximation given by RCF convergents is, in essence, best possible for almost all real numbers.

\begin{thm}\label{thm:Legendre}[Legendre's Theorem, Theorem 19 in \cite{Khinchin}]
If there exist relatively prime integers $p$ and $q$ with $q>0$ and
\(
\left| x- \frac{p}{q} \right| < \frac{1}{2q^2},
\)
then the fraction $p/q$ must be a convergent for $x$ (i.e., $p/q=p_n/q_n$ for some $n$). The constant $1/2$ here is known as the Legendre constant and is best possible.
\end{thm}

\begin{thm}\label{thm:Khinchin}[Khinchin's Theorem, Theorem 32 in \cite{Khinchin}]
Suppose $\phi:[1,\infty)\to [0,\infty)$ is a function such that $x \phi(x)$ is non-increasing.  Then for almost all real numbers $x$, there are infinitely (respectively, finitely) many solutions in integers $p,q$ to the diophantine inequality
\(
\left| x- \frac{p}{q} \right| < \frac{\phi(q)}{q}
\)
if the integral
\(
\int_1^\infty \phi(x) \ dx
\)
diverges (resp., converges).
\end{thm}

Khinchin's Theorem implies that for any $\epsilon>0$ and $C>0$ there are only finitely many solutions in relatively prime integers $p,q,$ to the inequality
\(
\left| x- \frac{p}{q} \right| < \frac{C}{q^{2+\epsilon}}
\)
for almost all $x$. So the exponent of $2$ in Hurwitz's Theorem is also best possible.

\begin{remark}
In higher dimensions, we consider $x=(x_1,x_2,\dots,x_n)\in \mathbb{R}^n$, and say a vector $(q,p_1,\dots,p_n)\in \mathbb{Z}^{n+1}$, $q\ge 1$ is an approximation of order $\eta$ to $x$ if \( \left| x_i - \frac{p_i}{q}\right| < q^{-\eta}, \quad 1\le i \le n.\) We let $\eta(x)$ denote the supremum over all $\eta$ for which the above inequality has infinitely many solutions. For almost all $x$, we have $\eta(x)=(n+1)/n$. \cite[Theorem 46]{schweiger}
\end{remark}

All the theorems above imply that convergents are of roughly the best \emph{order} of approximation, but in fact, the strength of continued fractions goes even deeper. We say a fraction $p/q$ in lowest terms is a \emph{best approximation of the first kind} to an irrational number $x$ if
\(
\left| x- \frac{p}{q} \right| < \left| x- \frac{p'}{q'} \right|  \text{ for all } p',q'\in \mathbb{Z}\text{ with }  q' <q.
\)
and a \emph{best approximation of the second kind} if
\(
|qx-p|< |q'x-p'| \text{ for all } p',q'\in \mathbb{Z}\setminus\{0\}\text{ with }  q' <q.
\)
Note that being a best approximations of the second kind is a stronger property than being a best approximation of the first kind.

\begin{thm}[See Theorem II in \cite{Cassels}]
For an irrational number $x$, best approximations of the second kind are also RCF convergents and vice-versa, with at most one exception.
\end{thm}

Best approximations of the \emph{first} kind are either RCF convergents or their intermediate fractions (also know as mediants or semi-convergents). For more information and definitions, see Chapter II of \cite{Khinchin}.

\subsection{New results}

We shall state the new results in terms of a more general domain $K$ instead of just the Dirichlet domain $K_D$ given in the opening section (see next section for more definitions). We define the radius of $K$ by
\(
\rad(K) := \sup \{\Norm{h}: h  \in K\}.
\)
All domains we will be interested in satisfy $\rad{K}<1$. It was shown in \cite{LV} that  $\rad(K_D)=2^{-1/4}$.

We define distance between two points $h_1=(u_1,v_1),h_2=(u_2,v_2)\in \Sieg$ by 
\[\label{eq:distancedefined}
d(h_1,h_2) := \Norm{ h_1^{-1} * h_2}=\norm{\overline{v_1}-\overline{u_1}u_2+v_2}^{1/2}.
\]

We begin with an analog of \ref{eq:RCFapprox}.

\begin{thm}\label{thm:truediophantineexponent} Let $h \in \Sieg$ be a point with at least $n+1$ continued fraction digits, including $\gamma_0$. 
We have
\(
d\left( \left( \frac{r_n}{q_n}, \frac{p_n}{q_n} \right), h \right) \asymp \left| \frac{v_{n+1}}{q_n^2} \right|^{1/2} 
\)
with implicit constant \(
\mathcal{R}_K :=  \prod_{n=1}^\infty \left( 1+ \rad(K)^n \right)^2 ,
\)
 and in particular,
\[\label{eq:distanceconstant}
d\left( \left( \frac{r_n}{q_n}, \frac{p_n}{q_n} \right), h \right) \le \frac{\rad(K) \cdot \mathcal{R}_K}{|q_n|}.
\]
\end{thm}

By saying ``$A \asymp B$ with implicit constant $C$,'' we mean that $C^{-1} A \le B \le C A$. The implicit constant given in Theorem \ref{thm:truediophantineexponent} appears to be very far from the truth. For the Dirichlet region, the constant $\rad(K_D) \mathcal{R}_{K_D}$ is approximately $ 5656.5$, but in Mathematica testing, we never saw a constant larger than $1.26$.

While we may not know yet what the best possible constant multiplier is in \eqref{eq:distanceconstant}, we can prove the following Khinchin-type result to show the exponent in \eqref{eq:distanceconstant} is best possible.

\begin{thm}\label{thm:diophantineexponent}
Let $\epsilon >0$ and $C>0$; and let $A\subset \Sieg$ denote the set of $h\in \Sieg$ such that
\(
 d\left( \left( \frac{r}{q}, \frac{p}{q} \right), h \right) \le \frac{C}{|q|^{1+\epsilon}} \text{ for infinitely many points }  \left( \frac{r}{q}, \frac{p}{q} \right) \in \mathcal{S}(\Q) ,
\)
where we assume that all rational points are in lowest terms.  
Then the set $A$ has measure $0$ with respect to the inherited measure on $\Sieg$ (see Section \ref{section:background}).
\end{thm}

As in Khinchin's result, the result holds for a more general function than $C|q|^{-1-\epsilon}$. We shall remark further about this in the proof.

Theorem \ref{thm:diophantineexponent} is similar to a result of Hersonsky and Paulin \cite{HersonskyPaulin}. We shall briefly discuss the relation between the two results in Section \ref{section:HersonskyPaulin}.

At this point, we are unable to prove an analog of Legendre's Theorem, but we can show that the convergents satisfy a notion of best approximation.

\begin{thm}\label{thm:bestapproximant}
Let $h \in \Sieg\setminus \Sieg(\Q)$ be an irrational point with $n$th convergent $\left(\frac{r_n}{q_n}, \frac{p_n}{q_n}  \right)$.  Let $\left( \frac{R}{Q}, \frac{P}{Q} \right)$ be a different rational point in $\mathcal{S}$ in lowest terms.

If 
\(
|Q| < \frac{1}{2\rad(K)^2 \mathcal{R}_K^2} |q_n|,
\)
 then 
\(
d\left( \left(\frac{R}{Q},\frac{P}{Q}\right), h \right) > d \left(\left( \frac{r_n}{q_n}, \frac{p_n}{q_n}\right), h \right).
\)
\end{thm}

We shall show a more precise statement than Theorem \ref{thm:bestapproximant} in Proposition \ref{prop:bestapproximant}. 

While the method of proof of Theorem \ref{thm:diophantineexponent} is similar to classical methods, we are not aware of a proof that resembles the one we use for Theorem \ref{thm:bestapproximant}.

\section{Background}\label{section:background}

We require additional facts and definitions about the continued fraction algorithm on the Heisenberg group. Many of the calculations that are skimmed over here are included in \cite{LV}.

To begin with, many readers may have seen the Heisenberg group in the form $(z,t)\in \Heis=\C\times\R$ with group action \( (z,t)*(z',t') = (z+z',t+t'+2\Im(z\overline{z'})). \) This can be translated to the Siegel model $\Sieg$ by \( (z,t) \in \Heis \mapsto (z(1+\ii),\norm{z}^2+t \ii) \in \Sieg.\) 

Note that Lebesgue measure on $\Heis$ (seen as $\R^3$) is left invariant under left-translation by elements of $g$. Therefore we will consider the default measure on $\Sieg$ to be the measure inherited from Lebesgue measure on $\Heis$. The volume of a ball of radius $r$ is $C r^4$, where $C$ equals the Lebesgue measure of the ball of radius $1$. We will denote the volume of a ball of radius $r$ by $\operatorname{Vol}(B_r)$.

We will use $h=(u,v)$ to denote a point in $\Sieg$, often, but not always, irrational. We will use $\gamma$ to denote an integer point in $\Sieg(\Z)$.

We are also interested in identifying $\Sieg$ as we defined it in the introduction---called the \emph{planar Siegel model}---with a subset of $\C^3$, considered projectively. In particular, we will consider the \emph{projective Siegel model} of $\Sieg$ as being
\[\label{eq:projectivemodel}
\left\{(z_1:z_2:z_3) \in \C^3\setminus \{(0:0:0\} : |z_2|^2 - 2 \Re(\overline{z_1}z_3)=0 \right\},
\]
where we consider two points in this space as being the same if one is a non-zero multiple of the other. We then identify the planar and projective models in the following way: $(u,v) \leftrightarrow (1:u:v)$. The projective model contains a point at infinity $(0:0:1)$ which is not in the planar model, but we will ignore this as it will not be important in the sequel. We also identify projective points $(1:u:v)$ with vertically written vectors:
\(
\left( \begin{array}{c}1\\u\\v \end{array} \right).
\)

We will freely swap back and forth between the two models for the remainder of the paper. One advantage given by the projective model is that it allows us to consider rational points in the planar model $(r/q,p/q)$ as ``integer'' points $(q:r:p) \in \Z[\ii]^3$ in the projective model.

The second, and most significant, advantage of the projective model is that it allows us to more simply model the action of linear fractional transformations on $\Sieg$.

Consider the group $U(2,1;\mathbb{Z}[\ii])$ of matrices given by
\(
\left\{ M\in GL(3,\mathbb{Z}[\ii]) \mid J M^{\dag}J = M^{-1}  \right\}
\)
where $M^{\dag}$ denotes the complex transpose of $M$ and 
\(
J = \left( \begin{array}{ccc}
0 & 0 & -1\\
0 & 1 & 0\\
-1 & 0 & 0
\end{array}  \right).
\)
We will allow these matrices to act on points in $\Sieg$ in the following way: given a point $h=(u,v)\in \Sieg$ and a matrix $M\in U(2,1;\Z[\ii])$, we let $Mh$ denote the point whose projective coordinates are given by
\(
M\left( \begin{array}{c}1\\u\\v \end{array} \right) .
\)

It can be shown that $J h = \iota h$, as such we may appropriately call $J$ the inversion matrix. Similarly, given a point $h=(u,v)$, let $T_h$ denote the matrix
\(
\left( \begin{array}{ccc} 1 & 0 & 0 \\ u & 1 & 0 \\ v & \overline{u} & 1 \end{array}\right);
\)
then the point $T_h h'$ is equivalent to $h* h'$. Therefore, we think of $T_h$ as a translation matrix. This also implies that $T_h^{-1}=T_{h^{-1}}$.

As we mentioned in the introduction, we want to consider more general sets $K$, beyond the Dirichlet domain $K_D$. We want to consider sets $K$ that satisfy the following properties:
\begin{enumerate}
\item $K$ is a fundamental domain for $\Sieg$ under the action of left-translation by $\Sieg(\Z)$; i.e., $\bigcup_{\gamma \in \Sieg(\Z)} \gamma* K = \Sieg$, and $\gamma* K \cap K =\emptyset$ for all $\gamma \in \Sieg(\Z)\setminus \{(0,0)\}$.
\item $\rad(K)<1$.
\item The boundary of $K$ is piece-wise smooth.
\end{enumerate}
The last condition is perhaps not necessary for the results of this paper, but avoids some $K$ that are not nice to work with. We will assume any $K$ considered throughout this paper satisfies all these conditions.

Given $K$, we define the nearest integer function with respect to $K$ by $[h]_K=\gamma$ if $h \in \gamma * K$ for $\gamma \in \Sieg(\Z)$. By definition of $K$, the integer point $\gamma$ is unique. When $K$ is known, we will often just write $[h]$ instead of $[h]_K$. 

With these definitions, we may define the forward iterates, the continued fraction digits, the convergents, and the Gauss map with respect to an arbitrary $K$ in the same way that we gave them for $K_D$ in Section \ref{sec:firstintro}. We will often write out convergents in projective form as $(q_n:r_n:p_n)$. In the same way that we would write the coordinates of $h$ as $(u,v)$, we write out the coordinates of $h_i$ as $(u_i,v_i)$.

With these definitions, we can write out the action of the Gauss map $T$ using the matrices we have seen before. For $h  \in K \setminus\{(0,0)\}$, we have
\(
h_1 = Th = T_{\gamma_1^{-1}} J h = T_{\gamma_1}^{-1} J h.
\)
Rewriting and then extending this, we obtain
\(
h =J T_{\gamma_1} h_1 = JT_{\gamma_1} JT_{\gamma_2} h_2 = \dots = JT_{\gamma_1} \dots JT_{\gamma_n} h_n.
\)
We will shorthand the matrix $JT_\gamma$ by $A_\gamma$ and $JT_{\gamma_1} \dots JT_{\gamma_n}$ by $Q_n$. Note that $A_{\gamma_n}h_n = h_{n-1}$ and $Q_n h_n = h$. The matrix $Q_n$ is closely related to the convergents $(q_n:r_n:p_n)$ in that we can write the matrix as
\(
Q_n = \left( \begin{array}{ccc}
q_n & \frak{q}_n & -q_{n-1}\\
r_n & \frak{r}_n & -r_{n-1}\\
p_n & \frak{p}_n & -p_{n-1}
\end{array}  \right) .
\)
The column $(\frak{q}_n:\frak{r}_n:\frak{p}_n)$ is completely determined by the other two columns and the matrix relation that defines $U(2,1;\Z[\ii])$.

We note that the distance between two points is unchanged if both points are left-translated by the same $\gamma$, i.e.
\[\label{eq:translated}
 \operatorname{d}(h,h') = \operatorname{d}(\gamma * h, \gamma * h').
\]
 Moreover, we have
\[\label{eq:inversed}
\operatorname{d}(h,h') = \Norm{h} \Norm{h'} \operatorname{d}(\iota h, \iota h').
\]

\begin{prop}[Lemma 3.19 and Theorem 3.23 in \cite{LV}]\label{prop:distance}
We have
\(
\operatorname{d}\left( \left( \frac{r_n}{q_n}, \frac{p_n}{q_n} \right), h \right) &=\left| \frac{\prod_{i=0}^n v_i}{q_n} \right|^{1/2}\\
&=  \left| \frac{1}{\overline{q_n}(q_{n+1}+\frak{q}_{n+1}u_{n+1}-q_nv_{n+1})} \right|^{1/2}.
\)
\end{prop}

We make the relationship between the last two equalities more explicit via the following results, one of which we shall prove below.

\begin{lemma}[Lemma 3.19 in \cite{LV}]\label{lemma:fracq}
We have
\(
q_n+\frak{q}_n u_n - q_{n-1} v_n=\frac{(-1)^n}{vv_1 v_2 \dots v_{n-1}}. 
\)
\end{lemma}

\begin{thm}\label{thm:relativesize}
We have
\(
|q_n+\frak{q}_n u_n - q_{n-1} v_n| = |vv_1 v_2 \dots v_{n-1}|^{-1} \asymp q_n
\)
with implicit constant $\mathcal{R}_K$.
\end{thm}

For the Dirichlet Region, the implicit constant is $\approx 6726.7$. Mathematica calculations thus far have only given a lower bound of $\approx 0.35$ and an upper bound of $\approx 3.38$.

By applying Theorem \ref{thm:relativesize} together with Lemma \ref{lemma:fracq} for two successive indices, we obtain the following corollary.

\begin{cor}\label{cor:successivesize}
We have
\(
|q_{n-1}| \asymp |v_n \cdot q_n|
\)
with implicit constant $\mathcal{R}_K^2$.
\end{cor}

We also have the following result.

\begin{lemma}[Equation (3.7) in \cite{LV}]
We have
\[\label{eq:prq}
\overline{p_n}-\overline{r_n}u + \overline{q_n} v = (-1)^n \prod_{i=0}^n v_i.
\]
\end{lemma}

We can prove, using an identical method, that
\[\label{eq:tildeprq}
\overline{\frak{p}_n}-\overline{\frak{r}_n}u + \overline{\frak{q}_n} v = (-1)^{n-1} u_n \cdot \prod_{i=0}^{n-1} v_i.
\]

In addition to the $\asymp$ notation introduced earlier, we will also make use of the big-O notation. We say $f(x)=O( g(x))$ if there exists a constant $C>0$ such that  $|f(x)|< C|g(x)|$ for all $x$ under consideration.

\section{A brief digression on Hersonsky and Paulin}\label{section:HersonskyPaulin}

Hersonsky and Paulin investigated the Heisenberg group and diophantine approximation from a different perspective. They looked at the Heisenberg group modeled by $\Heis= \C \times \R$, not $\Sieg$.
They examine rational points given by
\(
h= \left( \frac{a+b \ii}{q_1} , \frac{c}{q_2} \right) \in \mathbb{Q}[\ii] \times \mathbb{Q}, \qquad a,b,c,q_1,q_2 \in \mathbb{Z}
\)
in lowest terms, along with a height function $H(h)=\operatorname{lcm}(q_1,q_2)$. We say a point is irrational if it cannot be represented in the above form.

Hersonsky and Paulin proved the following Khinchin-type theorem.

\begin{thm}[Theorem 3.5 in \cite{HersonskyPaulin}]\label{thm:HP}
Let $\psi:\mathbb{R}^+ \to \mathbb{R}^+$ be a slowly varying map\footnote{For full definitions, see \cite{HersonskyPaulin}.}.  Let $E_\psi$ be the set of irrational points $h\in \Heis$ such that there exist infinitely many rational points $h'\in \mathbb{Q}[\ii] \times \mathbb{Q}$ with 
\(
\operatorname{d}(h,h') \le \frac{\psi (H(h'))}{H(h')}.
\)
Then, $\lambda(E_\psi)=0$ (or $\lambda(E_\psi^c)=0$), where $\lambda$ is the Lebesgue measure, if and only if the integral  
\(
\int_1^\infty \psi(t)^2 \frac{dt}{t}
\) converges (resp., diverges).
\end{thm}

In particular, there are infinitely many solutions for almost all points $h$ if $\psi(t)=1$ but not if $\psi(t)=t^{-\epsilon}$, with $\epsilon>0$.

We emphasize that Theorem \ref{thm:HP} and Theorem \ref{thm:diophantineexponent} are disjoint results. If we start with a rational point $h=(r/q, p/q) \in \mathcal{S}$, this corresponds to a rational point
\(
\left( \frac{r\cdot \overline{q}}{|q|^2}, \frac{\Im(p\overline{q})}{|q|^2} \right) \in \mathcal{H},
\)
which could have a height as large as $|q|^2$, or possibly smaller if there is some cancellation. We know, by Theorem \ref{thm:truediophantineexponent}, that for any irrational point $h\in\mathcal{S}$ the convergents satisfy
\(
d\left( \left( \frac{r_n}{q_n}, \frac{p_n}{q_n} \right), h \right) \le \frac{C}{|q_n|},
\)
for some $C>0$.  Theorem \ref{thm:HP} then merely says that if we were to convert these convergents into points in $\mathcal{H}$, the height must almost always be larger than, say, $|q_n|^{1-\epsilon}$.

\section{Proof of Theorem \ref{thm:relativesize}}

Recall from Lemma \ref{lemma:fracq} that
\(
q_n + \frak{q}_n u_n - q_{n-1} v_n = \frac{(-1)^n}{vv_1 v_2 \dots v_{n-1}}.
\)

We want to find a similar formula for $q_n$ itself.  Define
\[\label{eq:poweri}
(q_n^{(i)}:r_n^{(i)}: p_n^{(i)})
:=
A_{\gamma_{i+1}} A_{ \gamma_{i+2}} \dots A_{\gamma_n} (1:0:0),
\]
and let $v'_i=p_n^{(i)}/q_n^{(i)}$.  Equation \eqref{eq:poweri} and the definition of $A_\gamma$ also imply that $q_n^{(i)}= -p_n^{(i-1)}$ and that $p_n^{(n-1)}= -1$, so that
\(
q_n^{(i)} = \frac{(-1)^{(n-i)}}{v'_i v'_{i+1} \dots v'_{n-1}}, \qquad 0 \le i \le n-1
\)
Since $Q_n =A_{\gamma_1}A_{\gamma_2}\dots A_{\gamma_n}$, we have $q_n^{(0)}=q_n$, $p_n^{(0)}=p_n$ and $r_n^{(0)}=r_n$, so that
\(
q_n = \frac{(-1)^n}{v_0' v'_1 v'_2 \dots v'_{n-1}}.
\)

By \eqref{eq:poweri}, we see that $(q_n^{(i)}:r_n^{(i)}: p_n^{(i)})$ is the $n-i$th convergent to the point $T^i h$, since the first $n-i$ digits of $T^i h $ are precisely $\{\gamma_{i+1}, \gamma_{i+2},\dots \gamma_n\}$. By Proposition \ref{prop:distance}, we therefore have for $0\le i \le n-1$
\[\label{eq:distancelater}
\operatorname{d}(h_i,h_i') =  \left| \frac{v_i \cdot v_{i+1}\cdot v_{i+2} \cdots v_{n}}{q_n^{(i)}} \right|^{1/2}
\]
so that
\(
\frac{\operatorname{d}(h_i,h_i')}{\Norm{h_i}} &= \frac{\operatorname{d}(h_i,h_i')}{|v_i|^{1/2}}\\
&= \left| \frac{v_{i+1}\cdot v_{i+2} \cdots v_{n}}{q_n^{(i)}} \right|^{1/2}\\
&\le \operatorname{rad}(K)^{n-i} .
\)
where the last inequality follows because $|v_i|^{1/2}=\Norm{h_i} \le \operatorname{rad}(K)$ and $\norm{q_n^{(i+1)}}\ge 1$ (see \cite[Lemma~3.20]{LV}).
Combining \eqref{eq:translated}, \eqref{eq:inversed}, and \eqref{eq:distancelater}, we have
\(
\frac{\operatorname{d}(h_i,h_i')}{\Norm{h'_i}} &=\Norm{h_i}  \cdot \operatorname{d}(\iota h_i,\iota h_i')= \Norm{h_i} \cdot \operatorname{d}(\gamma_{i+1}^{-1}* \iota h_i,\gamma_{i+1}^{-1}* \iota h_i') \\
&=  \Norm{h_i}  \cdot \operatorname{d}( h_{i+1}, h_{i+1}')= |v_i|^{1/2}  \cdot \operatorname{d}( h_{i+1}, h_{i+1}')\\
&= \left| \frac{v_{i}\cdot v_{i+1} \cdots v_{n}}{q_n^{(i+1)}} \right|^{1/2}\\
&\le \operatorname{rad}(K)^{n-i+1} ,
\)

In particular, this gives
\(
\left| \frac{v'_i}{v_i} \right| &\le \left( \frac{\Norm{h_i'}}{\Norm{h_i}} \right)^2 = \left( \frac{\Norm{h_i}+ d(h_i,h_i')}{\Norm{h_i}} \right)^2 \\
&\le \left( 1+  \operatorname{rad}(K)^{(n-i)}  \right)^2
\)
and likewise
\(
\left| \frac{v_i}{v'_i} \right| \le \left( 1+  \operatorname{rad}(K)^{(n-i+1)}  \right)^2 \le  \left( 1+  \operatorname{rad}(K)^{(n-i)}  \right)^2
\)
Since $1+x \le e^x$, we have that
\(
\left| \prod_{i=0}^{n-1} \frac{v'_i}{v_i} \right| &\le \prod_{i=0}^{n-1} \left( 1 +\operatorname{rad}(K)^{(n-i)}  \right)^2 \le \exp\left( 2\sum_{i=0}^{n-1} \operatorname{rad}(K)^{(n-i)}  \right)\\
&< \exp\left(  2\sum_{i=1}^\infty \rad(K)^{i} \right)= \exp\left( \frac{2\rad(K)}{1-\rad(K)}\right)
\)
so that the product is uniformly bounded away from $\infty$ for \emph{all} $n$. (A similar argument applied to the reciprocal of this fraction shows that it must be bounded away from $0$ as well.)

Hence, we have
\(
\left| q_n + \frak{q}_n u_n - q_{n-1} v_n \right| = \left| \frac{1}{v_0 v_1 v_2 \dots v_{n-1}} \right| \asymp  \left| \frac{1}{v'_0 v'_1 v'_2 \dots v'_{n-1}} \right| = |q_n|
\)
with implicit constant 
\(
\mathcal{R}_K= \prod_{n=1}^\infty \left( 1+ \rad(K)^n \right)^2 
\)
 as desired.

\section{Proof of Theorem \ref{thm:truediophantineexponent}}

By Proposition \ref{prop:distance} and Lemma \ref{lemma:fracq}, we have
\(
\operatorname{d}\left( \left( \frac{r_n}{q_n}, \frac{p_n}{q_n} \right), h \right) = \left| \frac{1}{\overline{q_n}(q_{n+1}+\frak{q}_{n+1}u_{n+1}-q_nv_{n+1})} \right|^{1/2} = \left| \frac{v_n}{\overline{q_n}(q_{n}+\frak{q}_{n}u_{n}-q_{n-1}v_{n})} \right|^{1/2}.
\)
Using Theorem \ref{thm:relativesize} and the fact that $|v_n|^{1/2} \le \rad(K)$ completes the proof of this theorem.

\section{Proof of Theorem \ref{thm:diophantineexponent}}

Recall that we want to show that the set $A$ defined by the set of points $h\in \mathcal{S}$ such that
\(
 d\left( \left( \frac{r}{q}, \frac{p}{q} \right), h \right) \le \frac{C}{|q|^{1+\epsilon}} \text{ for infinitely many points }  \left( \frac{r}{q}, \frac{p}{q} \right) \in \mathcal{S}(\Q)
\)
in lowest terms, has measure $0$, for some fixed $C, \epsilon>0$.

In the interest of generality, let us instead consider a positive function $\phi(m)$ with $\limsup \phi(m) = 0$, and have $A$ be defined as the set of points  $h\in \mathcal{S}$ such that
\(
 d\left( \left( \frac{r}{q}, \frac{p}{q} \right), h \right) \le \phi(|q|^2) \text{ for infinitely many points }  \left( \frac{r}{q}, \frac{p}{q} \right) \in \mathcal{S}(\Q)
\)
in lowest terms. We want to know what conditions on $\phi(m)$ will cause $A$ to have measure $0$.

To show this, we will make use of a Borel--Cantelli type construction as well as the fact that $\operatorname{Vol}(B_r)=r^4$. 

It suffices to show that $A_K = A \cap K$ has measure $0$, since 
\(A = \bigcup_{\gamma \in \mathcal{H}(\mathbb{Z})} \gamma * A_K\)
would then be a countable union of measure $0$ sets and hence have measure $0$ itself. For $m\in \mathbb{N}$, let $A_m$ denote the set of points $h \in K$ such that there exists a rational point $\left(\frac{r}{q}, \frac{p}{q} \right)$ in lowest terms with $|q|^2 = m$ and 
\(
\operatorname{d} \left( h,  \left(\frac{r}{q}, \frac{p}{q} \right) \right) \le \phi(m).
\)
By construction, we have
\(
A_K = \bigcap_{M=1}^\infty \bigcup_{m \ge M} A_m.
\)

For any $M \in \mathbb{N}$, we have
\(
\mu\left( A_K \right) \le \sum_{m \ge M}\mu\left( A_m \right),
\)
where $\mu$ denotes the inherited measure on $\Sieg$. So if the sum of $\mu(A_m)$ converges, then we can choose $M$ large enough to make the sum on the right-hand side as small as we like, and therefore would have that $A_K$ has measure $0$.

It suffices to show that $\mu(A_m)$ is bounded by a function whose sum converges. We shall assume a worst case scenario: $A_m$ consists of \emph{disjoint} balls or radius $\phi(m)$ around all rational points $\left( \frac{r}{q}, \frac{p}{q} \right)$ with $|q|^2=m$ in a larger set $K'$ with the rational points \emph{not necessarily in lowest terms}. This larger set is the set of all points within distance $\sup_{m\in \mathbb{N}} \phi(m)$ of the fundamental domain $K$. We must consider rational points in $K'$, because points in $K$ could be approximated by rational points outside of $K$. 

So we have that
\(
\mu(A_m)&\le  \operatorname{Vol} \left(B_{\phi(m)}\right) \cdot \#\left\{   \left( \frac{r}{q}, \frac{p}{q} \right)\in K' \mid |q|^2 = m\right\}\\
&= O\left( \phi(m)^4 \right) \cdot \#\left\{   \left( \frac{r}{q}, \frac{p}{q} \right)\in K' \mid |q|^2 = m\right\}
\)
We will focus on this latter term and try to obtain an upper bound on how many rational points are in $K'$ with $|q|^2=m$.

First, note that  $\left( \frac{r}{q}, \frac{p}{q}\right)$ must be in the set $K'$, which has bounded radius. So $r$ must a complex integer with norm bounded by $O(|q|)$. So for a fixed $m$  there are at most $O(m)$ possibilities for $r$.

We will suppose that $m$ and $r$ are fixed for a while and count how many possible choices of $q$ and $p$ there are. The key facts are that $K'$ is bounded, so $|p/q|=O(1)$, and that $|r|^2- 2 \operatorname{Re}(\overline{q}p)=0$ by \eqref{eq:projectivemodel}. If we write $q=a+b\ii$ and $p=c+d\ii$, then the relation becomes
\[\label{eq:acbd}
 |r|^2 =2( ac+bd).
\]

We now split into cases based on the structure of $q$.

\textbf{Case 1: $q$ is completely real or completely imaginary.}  If $a$ (or $b$) equals $0$, then we immediately know several things. First, $m$ must be a perfect square and $b$ (resp., $a$) must equal $\pm \sqrt{m}$. So there are at most $4$ possible values for $q$. Second, from \eqref{eq:acbd}, we see that $d$ (resp., $c$) has at most one possible value, but $c$ (resp., $d$) is free to vary provided $|p|=O(|q|)$. Thus there are at most $O(m^{1/2})$ possible choices for $p$. 

Therefore, there are at most $O(m^{1/2})$ total possibilities for $p$ and $q$ in this case.

\textbf{Case 2: $q$ is neither completely real nor completely imaginary and the GCD of $a$ and $b$ is $g$.} Since $a^2+b^2=m$ and the GCD of $a$ and $b$ is $g$, we have that the number of distinct $q$ satisfying these conditions is $r_2(m/g^2)$, where $r_2(n)$ is the number of ways of writing $n$ as the sum of two squares. It is well-known that $r_2(n)=O(n^\epsilon)$ for any fixed $\epsilon>0$.  Note as well that $g^2$ must divide $m$

To count the number of possible $p$ for any given $q$, we rewrite \eqref{eq:acbd} to put $c$ in terms of $d$:
\[\label{eq:abcd2}
c=\frac{1}{2a}|r|^2 - \frac{b}{a}d.
\]
For any integer $d$, there is clearly one integer $c$ satisfying the above equation, thus counting the number of $p$ is equivalent to counting the number of possible $d$.  More importantly, there is a valid solution to \eqref{eq:abcd2} only if $|r|^2-2bd$ is divisible by $2a$. Since $2g$ divides $|r|^2$, this implies that $d$ must fall into a particular residue class modulo $|a|/g$ in order for \eqref{eq:abcd2} to have a solution.

At the same time, we know that $c^2+d^2 \le R \cdot m$ for some constant $R$.  Combining this with \eqref{eq:abcd2}, we obtain
\(
R \cdot m \ge \left( \frac{|r|^2-2bd}{2a} \right)^2 +d^2.
\)
An application of the quadratic formula shows that any integer solutions to the above inequality must have $d$ in the following interval:
\(
\left[  \frac{b|r|^2-|a|\sqrt{4Rm^2-|r|^4}}{2m} ,  \frac{b|r|^2+|a|\sqrt{4Rm^2-|r|^4}}{2m} \right].
\)
This interval has length bounded by
\(
 \frac{|a|\sqrt{4Rm^2-|r|^4}}{m} \le 2|a|\sqrt{R}
\)
But as noted earlier, $d$ must be in a specific residue class modulo $|a|/g$, which implies that there at most $O(g)$ values that $p$ can take.

Thus, there are at most $O(g \cdot r_2(m/g^2))$ possible choices of $p$ and $q$ in this case.

Therefore, we have 
\[
\sum_{m=1}^\infty \mu(A_m) &=  \sum_{m \text{ is a perfect square}}O\left( \phi(m)^4 r_2(m) m^{3/2}\right)\notag \\ &\qquad  + \sum_{m=1}^\infty O\left( \phi(m)^4  m\cdot \sum_{g^2 \mid m} g \cdot r_2(m/g^2) \right). \label{eq:upperbound}
\]

Let $\phi(m)=C/m^{(1+\epsilon)/2}$ for any $C,\epsilon>0$, as in the statement of Theorem \ref{thm:diophantineexponent}, and recall that $r_2(n)=O(n^{\epsilon})$.

It is clear that the first sum
\(
 \sum_{m \text{ is a perfect square}} \phi(m)^4 r_2(m) m^{3/2}=  \sum_{\ell =1}^ \infty \phi(\ell^2)^4 r_2(\ell^2) \ell^{3}
\)
converges for this choice of $\phi(m)$, since each summand will be $O(\ell^{-1-2\epsilon})$ . For the second sum, note that we can exchange the order of summation
\[\label{eq:rearranged}
 \sum_{m=1}^\infty \phi(m)^4  m\cdot \sum_{g^2 \mid m} g \cdot r_2(m/g^2)  = \sum_{g=1}^\infty g \sum_{d=1}^\infty \phi(g^2 d)^4\cdot g^2 d \cdot r_2(d)
\]
provided that either of these sums converges (since all the terms are positive). Consider the interior sum on the right-hand side of \eqref{eq:rearranged}. Each summand is bounded by $O(g^{-2-4\epsilon} d^{-1-\epsilon})$, so the entire interior sum is bounded uniformly by $O(g^{-2-4\epsilon})$. Thus
\(
 \sum_{g=1}^\infty g \sum_{d=1}^\infty \phi(g^2 d)^4\cdot g^2 d \cdot r_2(d) = \sum_{g=1}^\infty O\left( g^{-1-4\epsilon} \right) = O(1).
\)
Thus the theorem is proved.

We can see that any function $\phi$ for which the sums in \eqref{eq:upperbound} are bounded would cause the corresponding set $A$ to have measure $0$.

\section{Proof of Theorem \ref{thm:bestapproximant}}

We shall prove the slightly more powerful result:
\begin{prop}\label{prop:bestapproximant}
Let $h\in \Sieg$ and let $\left( \frac{r_n}{q_n}, \frac{p_n}{q_n} \right)$ be its $n$th rational approximant.  Let $(Q: R: P) \in \mathbb{Z}[\ii]^3\cap \mathcal{S}$ such that
\(
\left| \overline{P} - \overline{R} u + \overline{Q} v \right| &= x_1 \left| \overline{p_n} - \overline{r_n} u +\overline{q_n} v \right|\\
\left| Q \right| &= x_2 \left| q_n \right|,
\)
then 
\(\sqrt{x_1}+\sqrt{x_2} \ge |v_n|^{-1}\mathcal{R}_K^{-1} .
\)
\end{prop}

\begin{proof}[Proof of Proposition \ref{prop:bestapproximant}]
Since $\operatorname{U}(2,1;\mathbb{Z}[\ii])$ is closed under taking inverses, there exists a point $(a:b:c)\in \mathbb{Z}[\ii]^3 \cap \mathcal{S}$ such that
\[\label{eq:abcrelation}
Q_{n+1} \left( \begin{array}{c} a \\ b \\ c \end{array} \right) = 
\left( \begin{array}{ccc}
q_{n+1} & \frak{q}_{n+1} & - q_{n}\\
r_{n+1} & \frak{r}_{n+1} & - r_{n} \\
p_{n+1} & \frak{p}_{n+1} & - p_{n}
\end{array}
\right) \left( \begin{array}{c} a \\ b \\ c \end{array} \right)  = \left( \begin{array}{c} Q \\ R \\ P \end{array} \right).
\]
From this relation, \eqref{eq:prq}, and \eqref{eq:tildeprq}, we can derive the equations
\[\label{eq:overlineQ}
Q = aq_{n+1}+b \frak{q}_{n+1} - c q_n,
\]
and
\(
\overline{P}-\overline{R}u+\overline{Q}v &=\overline{a} \left( \overline{p_{n+1}} - \overline{r_{n+1}}u +\overline{q_{n+1}}v\right) + \overline{b} \left( \overline{\frak{p}_{n+1}}-\overline{\frak{r}_{n+1}} u + \overline{ \frak{q}_{n+1}} v \right) \\
&\qquad- \overline{c} \left( \overline{p_n} -\overline{ r_n} u + \overline{q_n} v \right)\\
&= \overline{a}(-1)^{n+1}(v_0 v_1 \dots v_{n+1}) +\overline{b}(-1)^n (v_0v_1 \dots v_n u_{n+1})\\
&\qquad - \overline{c} (-1)^n(v_0 v_1 \dots v_n)\\
&= (-1)^n (v_0 v_1 \dots v_n) \left( -\overline{a}v_{n+1} + \overline{b} u_{n+1}-\overline{c}  \right) ,
\)
and, by applying \eqref{eq:prq} again, we get
\[\label{eq:fracQ}
\overline{P}-\overline{R}u+\overline{Q}v = (-1) \left( \overline{p_n} -\overline{ r_n} u + \overline{q_n} v \right) \left( \overline{a}v_{n+1}- \overline{b} u_{n+1}+\overline{c}  \right).
\]

If $a=0$, then $b$ must also equal $0$ in order for the point $(a:b:c)$ to be in the set $\mathcal{S}$. In this case, the conclusion of the proposition follows trivially from \eqref{eq:overlineQ} and \eqref{eq:fracQ}. Therefore we assume that $|a|\ge 1$ for the remainder of this proof.

From \eqref{eq:overlineQ} and \eqref{eq:distancedefined}, we have that
\(
x_2 &= |a| \cdot \left|   \left(\frac{q_{n+1}}{-q_n}\right) -\left(-  \frac{b}{a}\right) \cdot \left(\frac{\frak{q}_{n+1}}{-q_n}\right)+ \frac{c}{a}  \right|\\
&= |a| \cdot d\left(\left( \overline{\left( \frac{\frak{q}_{n+1}}{-q_n}\right) } ,  \overline{\left( \frac{q_{n+1}}{-q_n}\right) }  \right) ,  \left(-\frac{b}{a}, \frac{c}{a} \right) \right)^2.
\)
Note that the point $(-\overline{q_n}:\overline{\frak{q}_{n+1}}:\overline{ q_{n+1}})$ is a point in $\mathcal{S}$. In particular, $(-\overline{q_{n}}: -\overline{\frak{q}_{n+1}}: \overline{q_{n+1}})$ corresponds to the point $Q_{n+1}^{-1} (0:0:1) = J Q_{n+1}^\dagger J (0:0:1)$ and thus is in $\Sieg$; and if $(1:u:v)\in \Sieg$, then $(1:-u:v)\in \Sieg$ too.

 We also have, by the same argument
\(
x_1 =   |a| \cdot d\left(\left(-\frac{b}{a}, \frac{c}{a} \right),   \left( -u_{n+1},v_{n+1}  \right)  \right)^2
\)

But then by the triangle inequality,
\(
\frac{1}{|a|}\left(\sqrt{x_1}+ \sqrt{x_2} \right)
&=d\left( \left(-\frac{b}{a}, \frac{c}{a} \right) ,\left( \overline{\left( \frac{\frak{q}_{n+1}}{-q_n}\right) } , \overline{\left( \frac{q_{n+1}}{-q_n}\right) }  \right) \right)\\
&\qquad + d\left(\left(-\frac{b}{a}, \frac{c}{a} \right),   \left(-u_{n+1},v_{n+1}  \right)  \right)\\
&\ge d\left( \left( \overline{\left( \frac{\frak{q}_{n+1}}{-q_n}\right) } ,  \overline{\left( \frac{q_{n+1}}{-q_n}\right) }  \right) ,   \left( -u_{n+1},v_{n+1}  \right)  \right)\\
&= \left| v_{n+1} -\frac{\frak{q}_{n+1}}{q_n} u_{n+1} - \frac{q_{n+1}}{q_n} \right|^{1/2}\\
&= \left| \frac{q_{n+1}+\frak{q}_{n+1} u_{n+1} - q_n v_{n+1}   }{q_n} \right|^{1/2}.
\)
By Theorem \ref{thm:relativesize}, this gives
\(
\frac{1}{|a|}\left(\sqrt{x_1}+ \sqrt{x_2} \right) \ge \frac{1}{|v_{n}| \mathcal{R}_K}
\)

This completes the proof of the theorem.

\end{proof}

\begin{proof}[Proof of Theorem \ref{thm:bestapproximant}]
Suppose 
\(
|Q| < \frac{1}{2\rad(K)^2\mathcal{R}_K^2} |q_n|,
\)
then by Proposition \ref{prop:bestapproximant}, we would have $x_1/x_2> 1$.  Thus, 
\(
\operatorname{d}\left( \left( \frac{R}{Q}, \frac{P}{Q} \right), h \right) &= \left|\frac{ \overline{P} - \overline{R} u + \overline{Q} v}{Q} \right| ^{1/2}\\
&>  \left|\frac{ \overline{p_n} - \overline{r_n} u + \overline{q_n} v}{q_n} \right| ^{1/2} =\operatorname{d}\left( \left( \frac{r_n}{q_n}, \frac{p_n}{q_n} \right), h \right) ,
\)
as desired.\end{proof}

\section{Acknowledgements}

The author wishes to thank A. Lukyanenko for his continued insights.

\end{document}